\documentclass[reqno,oneside,a4paper,11pt]{amsart}
\usepackage[utf8]{inputenc}
\usepackage{mathtools} 

\usepackage{amsmath,amsthm,amsbsy,amssymb, comment, booktabs, float}
\usepackage{arydshln} 
\usepackage{transparent} 
\usepackage{physics} 
\usepackage{enumitem} 
\makeatletter
\def\namedlabel#1#2{\begingroup
    #2%
    \def\@currentlabel{#2}%
    \phantomsection\label{#1}\endgroup
}
\makeatother

\usepackage{xparse}     

\NewDocumentCommand{\mybar}{ O{0.860} O{0.3pt} m }{
    \mathrlap{\hspace{#2}\overline{\scalebox{#1}[1]{\phantom{\ensuremath{#3}}}}}\ensuremath{#3}
}

\newcommand{\cC}{\mathcal{C}}
\newcommand{\cF}{\mathcal{F}}

\newcommand{\cL}{\mathcal{L}}

\newcommand{\cM}{\mathcal{M}}

\newcommand{\cP}{\mathcal{P}}
\newcommand{\cQ}{\mathcal{Q}}

\newcommand{\cS}{\mathcal{S}}

\usepackage{mathrsfs}

\newcommand{\scL}{\mathscr{L}}

\makeatletter
\renewcommand{\pod}[1]{\allowbreak\mathchoice
  {\if@display \mkern 18mu\else \mkern 2mu\fi (#1)}
  {\if@display \mkern 6mu\else \mkern 2mu\fi (#1)}
  {\mkern4mu(#1)}
  {\mkern4mu(#1)}
}
\makeatother

\usepackage{bm}

\newcommand{\sdfrac}[2]{\mbox{\small$\displaystyle\frac{#1}{#2}$}}

\newcommand{\Sa}{\cS} 
\newcommand{\altshortminus}{\raisebox{-1pt}{\protect\scalebox{1.1}{-}}}
\newcommand{\Sm}{\cS_{\altshortminus}}

\newcommand{\z}{\operatorname{z}}

\newcommand{\dE}{\operatorname{d_E}}
\newcommand{\dP}{\operatorname{d_{pc}}}
\newcommand{\scBpc}{\mathscr{B}_{\mathrm{pc}}}
\DeclareMathOperator*{\diamE}{\operatorname{diam}_{\raisebox{-2pt}{\scriptsize E}\hspace*{-1.05pt}}}

\newcommand{\ZZ}{\mathbb{Z}}
\newcommand{\ZZp}{\mathbb{Z}_0^{+}}
\renewcommand{\ZZp}{\mathbb{Z}_{\ge 0}}

\newcommand{\NN}{\mathbb{N}}

\theoremstyle{plain}
\newtheorem{theorem}{Theorem}

\newtheorem{lemma}{Lemma}[section]

\newtheorem{problem}{Problem}[section]

\theoremstyle{remark}
\newtheorem{remark}{Remark}[section]

\theoremstyle{definition}

\newtheorem*{claim*}{Claim}

\usepackage[font={up,small}]{caption}
\usepackage{xcolor} 
\definecolor{orange}{rgb}{1,0.5,0}
\definecolor{Ggreen}{rgb}{0.,0.575,0.0128}
\definecolor{Bblue}{rgb}{0.016,.132,0.91}
\usepackage[hyphens]{url}

 \usepackage[backref=page]{hyperref}
\hypersetup{
    colorlinks = true,
linkcolor={red},
urlcolor={Bblue},
citecolor={Ggreen},
urlcolor = {Bblue},
citebordercolor = {0.33 .58 0.33},
 linkbordercolor = {0.99 .28 0.23},
 breaklinks=true
}

\renewcommand*{\backref}[1]{}
\renewcommand*{\backrefalt}[4]{%
  \ifcase #1 %
No citations.
  \or
(page #2).%
  \else
(pages #2).%
  \fi%
}

\usepackage{cite}
\usepackage[capitalise]{cleveref}
\crefname{lem}{Lemma}{Lemmas}
\crefname{rem}{Remark}{Remarks}
\crefname{thm}{Theorem}{Theorems}
\crefname{cor}{Corollary}{Corollaries}
\crefname{equation}{Equation}{Equations}
\crefname{prop}{Proposition}{Propositions}
\crefname{conj}{Conjecture}{Conjectures}
\crefname{section}{Section}{Sections}
\crefname{obs}{Observation}{Observations}
\crefname{claim}{Claim}{Claims}

\usepackage{subfigure}  
\usepackage[pdftex]{graphicx}
\graphicspath{{IMAGES/}{IMAGES12/}{./}}

\makeatletter
\def\mysequence#1{\expandafter\@mysequence\csname c@#1\endcsname}
\def\@mysequence#1{%
  \ifcase#1\or left\or right\or altceva\else\@ctrerr\fi}
\makeatother

\usepackage{tikz}
\usetikzlibrary{arrows.meta,    
                bending,        
                tikzmark}

\newcommand{\Qu}{\cQ_{\mathrm{I}}}
\newcommand{\Qd}{\cQ_\mathrm{II}}
\newcommand{\Qt}{\cQ_\mathrm{III}}
\newcommand{\Qp}{\cQ_\mathrm{IV}}

\renewcommand{\binom}[3][0.74]{\scalebox{#1}{$\dbinom{#2}{#3}$}}
\makeatletter
\@namedef{subjclassname@2020}{\textup{2020} Mathematics Subject Classification}
\makeatother
\begin{document}

\title[On the central ball in a translation invariant involutive field]
{On the central ball in a translation invariant involutive field}

\author[C. Cobeli, A. Raghavan, A. Zaharescu]
{Cristian Cobeli, Aaditya Raghavan, Alexandru Zaharescu}

\address[CC, AZ]{\normalfont 
``Simion Stoilow'' Institute of Mathematics of the Romanian Academy,~21 Calea Griviței Street, 
P. O. Box 1-764, Bucharest 014700, Romania}
\address[AR, AZ]{\normalfont
Department of Mathematics,
University of Illinois at Urbana-Champaign,
Altgeld Hall, 1409 W. Green Street,
Urbana, IL, 61801, USA\vspace{7pt}}

\email{cristian.cobeli@imar.ro}
\email{ar58@illinois.edu}  
\email{zaharesc@illinois.edu}  






\begin{abstract}
The iterated composition of two operators, both of which are involutions and translation invariant, partitions the set of lattice points in the plane into an infinite sequence of discrete parabolas. Each such parabola contains an associated stairway-like path connecting certain points on it, induced by the alternating application of the aforementioned operators. Any two lattice points in the plane can be connected by paths along the square grid composed of steps either on these stairways or towards taxicab neighbors. This leads to the notion of the parabolic-taxicab distance between two lattice points, obtained as the minimum number of steps of this kind needed to reach one point from the other.

In this paper, we describe patterns generated by points on paths of bounded parabolic-taxicab length and provide a complete description of the balls centered at the origin. In particular, we prove an earlier conjecture on the area of these balls.
\end{abstract}
\maketitle

\section{Introduction}
Inspired by the Ducci game~\cite{CPZ2016,CM1937} and the Proth-Gilbreath procedure 
(see~\cite{Odl1993,BCZ2024a} and the references therein),
which involves higher-order differences calculated recursively,
in a tessellation problem~\cite{BCZ2024b} regarding covering the plane with integers,
a particular type of operators that are linear in all their components 
are proven to be useful. 
Following~\cite{CZ2024}, their two-dimensional version is defined as follows: 
for any $(x,y)\in\ZZ^2$ let
\begin{align*}
    L'(x,y)& = (-x+2y+1,y),\
    \text{ and }\
    L''(x,y) = (x,2x-y+1).
\end{align*}
Among operators of this type that are linear in both variables 
and have arbitrary coefficients, $L'$ and $L''$ are the only ones 
that are both involutions and translation invariant.
Starting with an arbitrary point $P\in\ZZ^2$, 
the repeated application of $L'$ and $L''$ 
generates a sequence of points, which are the nodes of 
a stairway-like path situated alternately on the branches of a parabola.
These parabolas are parallel and the plane $\ZZ^2$ is completely covered by them.

In order to be able to jump from one parabola to another, 
let $M',M'',M''',M^{iv}$ be the operators that provide movement 
on the grid in all directions towards the nearest neighbors:
\begin{align*}
    M'(x,y)& = (x+1,y),& M''(x,y)& = (x-1,y),\\
    M'''(x,y)& = (x,y+1),& M^{iv}(x,y)& = (x,y-1).
\end{align*}
These four operators lead to the taxicab geometry initiated in 1952 
by Menger~\cite{Men1952}
(see {\c{C}}olako{\u{g}}lu~\cite{Col2019} and the references therein 
for more recent developments on a generalization of the taxicab distance).
From a geometric perspective, we now have two ways of moving in~$\ZZ^2$: 
one fast, following the steps of the stairways, and the other slow, 
moving only towards the nearest points in the grid.
Then, by combining in any order the operators from
$\cL\cup\cM$, 
where $\cL = \{L',L''\}$ and $\cM = \{M',M'',M''',M^{iv}\}$,
we see that there are infinitely many paths in $\ZZ^2$
from one point to another. 
Considering each step to have a length of $1$, 
whether it is on a parabola or towards a nearby neighbor point, 
we obtain the \textit{parabolic-taxicab distance} $\dP$, 
which is defined as the number of steps 
of the shortest path connecting two points.
Thus, for every $P,Q \in\ZZ^2$, let
\begin{align*}
    \dP(P,Q) &:= \min\{n\in\ZZ_{\ge 0} :
	F_1,\dots,F_n\in\{\mathrm{id}\}\cup\cL\cup\cM,\ 
	F_1\circ\cdots\circ F_n(P) = Q\}.
\end{align*}
Upon observing that $\dP$ satisfies the axioms of a distance,
for every point $C\in\ZZ^2$ and every non-negative integer $r$, 
we can define the \textit{parabolic-taxicab closed ball} 
with center $C$ and radius $r$ 
as the set of lattice points that are 
at a distance at most~$r$ from $C$, that is,
\begin{equation*}
    \scBpc(C,r) :=\big\{
    X\in\ZZ^2 : \dP(C,X) \le r
    \big\}.
\end{equation*}
Then, the \textit{perimeter} or the \textit{boundary} of 
$\scBpc(C,r)$ is
\begin{align*}
    \partial \scBpc(C,r) &:= \{(x,y) \in \ZZ^{2} : \dP(C, X) = r\}.
\end{align*}
A representation of these objects can be seen in Figure~\ref{FigureBallAndBorders}.
A tally on the small balls with center $O=(0,0)$ indicates that for $r\ge 0$ 
the number of points in $\scBpc(O,r)$ belongs to the sequence:
\begin{align*}
\scalebox{0.91}{
  1, 5, 15, 37, 75, 135, 221, 339, 493, 689, 931, 1225, 1575, 1987, 2465, 3015, 3641, 4349,  5143,\dots  }
\end{align*}
Then the sequence of the perimeters coincides with the sequence of gaps of the above 
sequence, except for $r=0$, where $\#\partial\scBpc(O,0)=1$. Thus, for $r\ge 0$, 
the number of points in $\partial\scBpc(O,r)=1$ generate the sequence:
\begin{align*}
\scalebox{0.91}{
1, 4, 10, 22, 38, 60, 86, 118, 154, 196, 242, 294, 350, 412, 478, 550, 626, 708, 794,\dots}
\end{align*}

\begin{figure}[hbt]
 \centering
 \hfill
     \includegraphics[height=\textwidth, angle=-90]{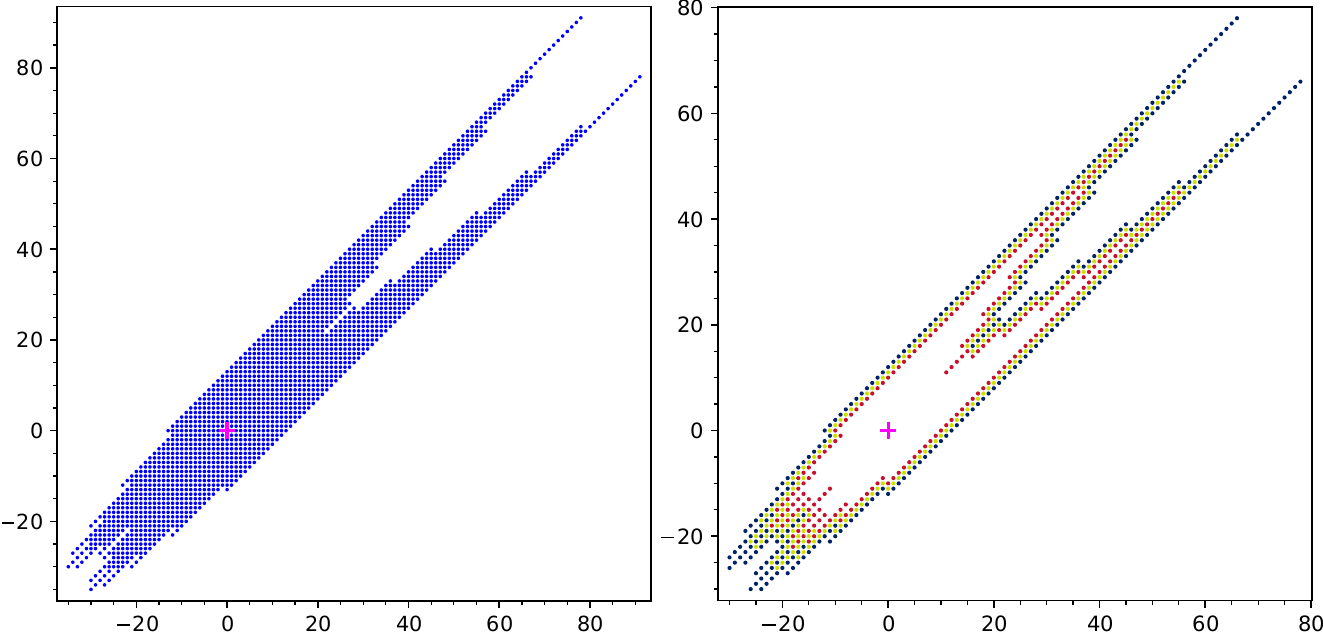}
    \hfill\mbox{}\\
\caption{In the image on the left-side $\scBpc(O,13)$ is shown, and in the image on the right-side the three consecutive borders of the balls with radii $10, 11$, and $12$ are shown. 
The cardinalities are: $\#\scBpc(O,13)=1987$,
$\#\partial \scBpc(O,10)= 242$, 
$\#\partial \scBpc(O,11)= 294$, and
\mbox{$\#\partial \scBpc(O,12)=350$.}
} 
\label{FigureBallAndBorders}
 \end{figure} 
 
Our aim here is to explore the atypical implications 
of measuring distances with $\dP$.
A series of tests with small radii show that parity plays an essential role 
in how the circles $\partial\scBpc(O,r)$ 
with $0\le r\le R$ add up in successive layers 
to form a complete ball $\scBpc(O,R)$. 
For instance, we noted that $\partial\scBpc(O,r)$ contains points 
on the diagonal $y=x$ if and only if~$r$ is even. 
Furthermore, we observed that $\partial\scBpc(O,r)$ 
does not contain points on the lines $y = x + |r-1|$ and also, 
alternatively, not on the parallel lines of integer $y$-intercept
that are increasingly closer to the diagonal $y=x$. 
Since the number of points added at step $r$ is as large as $O\big(r^2\big)$, 
when the parity shifts from one radius to the next, it causes the balls to develop
a pulsating characteristic as their radius increases.
These preliminary remarks 
indeed prove to be true for every $r\ge 0$, 
and they are part of the complete characterization of $\scBpc(O,r)$, 
as described in the following theorem.
\begin{theorem}\label{Theorem0}
Let $r\ge 0$ and let $c$ be integers.
Let $\Sa(r,c)$ be the set that contains the first coordinate of 
all points situated on both the boundary of $\scBpc(O,r)$ 
and the line $y=x+c$.
We denote the negative part of this set by $\Sm(r,c)=\Sa(r,c)\cap \ZZ_{<0}$.

\begin{enumerate}[wide, labelwidth=0pt, labelindent=0pt]
\setlength{\itemsep}{6pt}
\newlength{\myl}
  \settowidth{\myl}{\texttt{ I}}
 {\setlength\itemindent{\myl}
    \item[\namedlabel{Theorem0:I}{\normalfont{\texttt{I.}}}]
If $r$ and $c$ have opposite parity, or $|c| > r$, then $\Sa(r,c)=\emptyset$.
}
   \item[\namedlabel{Theorem0:II}{\normalfont{\texttt{II.}}}]
 Suppose $r$ and $c$ have the same parity 
and let $k$ be an integer defined by  $|c| = r - 2k$.
Then $\Sa(r,c)$ and $\Sm(r,c)$ are translations of intervals of 
integers, which can be explicitly expressed as follows:
\begin{enumerate}[topsep=6pt]
    \item[\namedlabel{Theorem0:IIa}{\normalfont{(\texttt{a})}}]
   If $|c| < r$ and $c \equiv r \pmod 2$, then:
\begin{equation}\label{eqChar0}
  \Sa(r,c) \cap \ZZp =
 \begin{cases}
  \left[\binom[0.81]{r-1-k}{2} + k, \binom[0.81]{r-k}{2} + k\right] \cap \ZZ,
    &\text{ if } 0 \leq c \leq r-2;\\[10pt]
 \left[\binom[0.81]{r-k}{2} + 1, \binom[0.81]{r-k+1}{2}\right] \cap \ZZ,
   &\text{ if } -(r-2) \leq c < 0.
 \end{cases}
\end{equation} 
  \item[\namedlabel{Theorem0:IIb}{\normalfont{(\texttt{b})}}]
If $-r+2 \le c \le 0$, then:
\begin{equation}\label{eqChar1}
  \Sm(r,c) 
  = \begin{cases}
    \big([k-r,-1] \cap \ZZ\big) + c(k-1) - \binom[0.81]{k}{2},  
    &\text{if } k \equiv 1 \pmod 2;\\[8pt]
    \big([0,r-k-1]\cap \ZZ\big) + ck - \binom[0.81]{k+1}{2}, 
    &\text{if } k \equiv 0 \pmod 2,
    \end{cases}
\end{equation} 
    \item[\namedlabel{Theorem0:IIc}{\normalfont{(\texttt{c})}}]
  If $0 < c \le r-2$, then:
\begin{equation}\label{eqChar2}
  \Sm(r,c) 
  = \begin{cases}
    \big([0,r-k-1] \cap \ZZ\big) - c(k+1) - \binom[0.81]{k+1}{2},  
              &\text{if } k \equiv 1 \pmod 2;\\[8pt]
    \big([k-r,-1] \cap \ZZ\big) - ck - \binom[0.81]{k}{2},
              &\text{if } k \equiv 0 \pmod 2.
    \end{cases}
\end{equation}  
     \item[\namedlabel{Theorem0:IId}{\normalfont{(\texttt{d})}}]
  Lastly, we have: 
\begin{equation}\label{eqChar3}
  \begin{split}
  \Sa(r,r) &= \left[-r,\binom[0.81]{r}{2}\right] \cap \ZZ ,\ \text{ and }\
  \Sa(r,-r) = \left[0,\binom[0.81]{r+1}{2}\right] \cap \ZZ .
  \end{split}
\end{equation}
\end{enumerate}
\end{enumerate}
\end{theorem}

Next, as a result, we obtain the area $\#\scBpc(O,r)$ 
and the perimeter $\#\partial\scBpc(O,r)$.
\begin{theorem}\label{Theorem1}
Let $r\ge 0$ be an integer.
Then the number of points on the boundary of the
closed parabolic-taxicab ball of radius $r$ and center $O=(0,0)$ is
\begin{equation}\label{eqTheorem1a}
  \#\partial \mathcal{B}(O,r) 
    = \frac{1}{2}\big(5r^2 - r\big) 
    -\left\lceil\frac{r-1}{2}\right\rceil
    +\left\lceil\frac{r}{r+1}\right\rceil
    + 1,
\end{equation}
and the number of points on the entire closed ball is
\begin{equation}\label{eqTheorem1b}
    \#\scBpc(O,r) = \sdfrac{1}{12}\big(10r^3 + 9r^2 + 23r\big) 
    + \frac{1}{2}\left\lceil\frac{r}{2}\right\rceil + 1.
\end{equation}
\end{theorem}
Theorem~~\ref{Theorem1} proves, in particular, the unfolded formula for the area of the parabolic-taxicab ball centered at the origin as conjectured in~\cite{CZ2024}. 
\section{Notation and Notes}
\subsection{Notation}
Let $\ZZp$ denote the set of non-negative integers
and let $O=(0,0)$ denote the origin of $\ZZ\times\ZZ$. 
The set of lattice points in each of the four quadrants 
of the real plane are:
\begin{align*}
    \Qu &:= \{(x,y) \in \ZZ^{2} : x \geq 0, y \geq 0\},&
    \Qd &:= \{(x,y) \in \ZZ^{2} : x \leq 0, y \geq 0\},\\
    \Qt &:= \{(x,y) \in \ZZ^{2} : x \leq 0, y \leq 0\},&
    \Qp &:= \{(x,y) \in \ZZ^{2} : x \geq 0, y \leq 0\}.\\  
\end{align*}
We use Minkowski's notation for the sum between a number and a set:
\begin{align*}
  a+\cM := \big\{a + x : x \in \cM\big\} =: \cM + a.
\end{align*}
We let $\scL(c)$ denote the set of lattice points that lie on the line $y=x+c$,
that is,
\begin{align*}
    \scL(c) &:= \{(x,y) \in \ZZ^{2} : y = x + c\}.
\end{align*}
Further, let $\cF(r,c)$ denote the set of points on the frontier of 
$\scBpc(O,r)$ that are aligned along the line $\scL(c)$, so that
\begin{align*}
    \cF(r,c) &:= \partial \scBpc(O,r) \cap \scL(c). 
\end{align*}
With these notations, the sets characterized in \cref{Theorem0}
are exactly the projections of~$\cF(r,c)$ onto the $x$-axis:
\begin{align*}
    \Sa(r,c)&:=\big\{x : (x,x+c)\in \cF(r,c)\big\},\\
\intertext{and}
    \Sm(r,c)&:= \Sa(r,c)\setminus \ZZp.
\end{align*}
Examples of such sets are shown in Figure~\ref{FigureSaSm}.
\begin{figure}[htb]
 \centering
 \hfill
    \includegraphics[height=\textwidth,angle=-90]{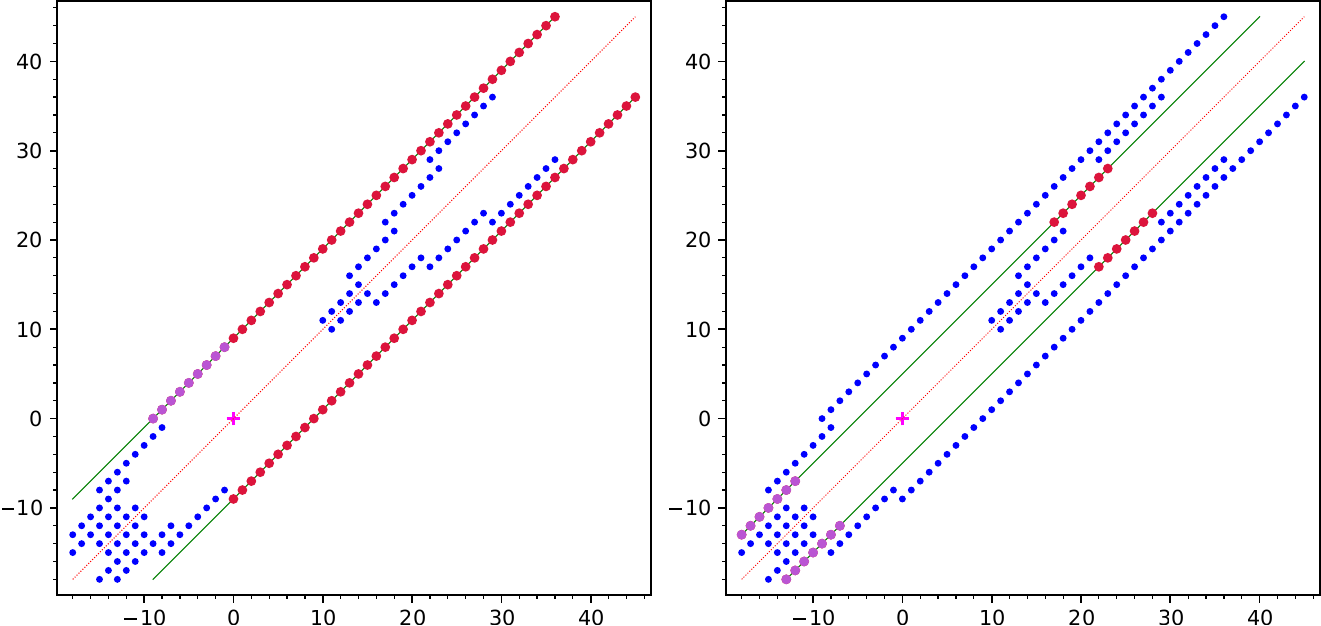}
    \hfill\mbox{}\\
\caption{Highlighted are the sets whose projections on the first coordinate are equal to $\Sa(r,c)$ and $\Sm(r,c)$ for $r=9$, $c =\pm 9$ (left) and $r=9$, $c =\pm 5$ (right).
Thus, we have: 
$\Sm(9,9)=[-9,-1]\cap\ZZ$, $\Sa(9,9)=[-9,36]\cap\ZZ$;
$\Sa(9,-9)=[0,45]\cap\ZZ$, $\Sm(9,-9)=\emptyset$;
$\Sm(9,5)=[-18,-12]\cap\ZZ$, $\Sa(9,5)=\big([-18,-12]\cup[17,23]\big)\cap\ZZ$;
$\Sm(22,-5)=[-13,-7]\cap\ZZ$, $\Sa(9,-5)=\big([-13,-7]\cup[22,28]\big)\cap\ZZ$.}
\label{FigureSaSm}
 \end{figure}
\subsection{Notes}
We include here a few general introductory remarks on the parabolic-taxicab distance
and the ball centered at the origin that it generates.
\subsubsection{The set of lattice points $\ZZ^2$ partitioned by parabolas, and the stairways induced within parabolas by alternating $L'$ and $L''$} 
Given a point $P\in\ZZ^2$ and a sequence of operators
$\{L_n\}_{n\ge 0}\subset\cL^{\NN}$, we obtain the sequence of points
$\cP=\{\big(L_1\circ\cdots\circ L_n\big)(P)\}_{n\ge 0}$.
According to~\cite[Theorem 2]{CZ2024}, all points in $\cP$ belong 
to a parabola of vertex $(m,m)$, where $m$ is the minimum of all coordinates of points in $\cP$.
These parabolas are disjoint, they are translations of one another along the first diagonal $y=x$ (see image on the left of Figure~\ref{FigureParabolas}),
and their union partitions the set of all lattice points in the plane.
%
\begin{figure}[htb]
 \centering
 \hfill
    \includegraphics[height=\textwidth,angle=-90]{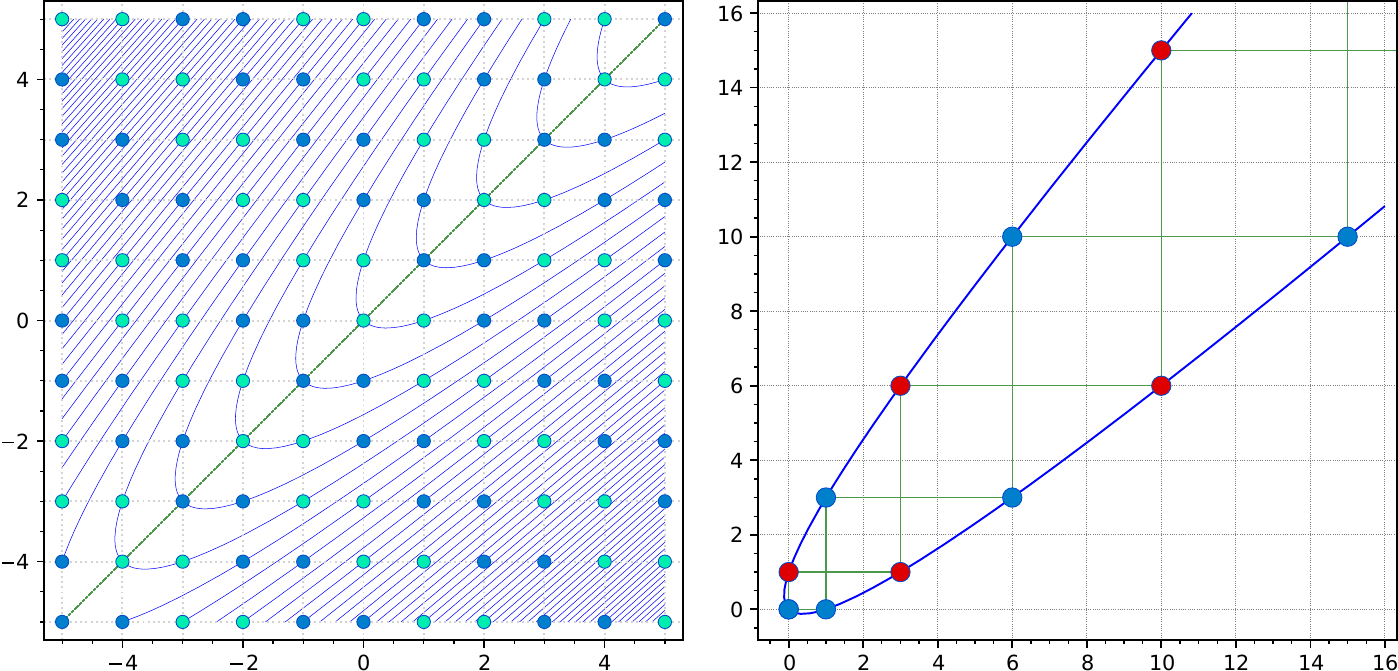}
    \hfill\mbox{}\\
\caption{The real plane partitioned by parabolas $x+y+2m=(x-y)^2$, with $m\in\ZZ$ (left).
Steps of length $1$ as measured by the parabolic-taxicab distance (right) by applying $L'$, then $L''$, (or vice-versa) and continuing to alternate. 
The indicated points have coordinates $(T_k,T_{k+1})$ and $(T_{k+1},T_{k})$, 
where $T_k=\binom{k+1}{2}$ are the triangular numbers. 
}
\label{FigureParabolas}
 \end{figure}
Also, since both $L'$ and~$L''$ are involutions, that is, 
$L'\circ L'= \mathrm{id} = L''\circ L''$,
let us remark that $\cP$ has no repetitions if and only 
$L'$ and $L''$ appear in the compositions
in alternating order.
Furthermore, the first application of either operator $L'$ or $L''$ 
in the sequence of compositions establishes the direction in which 
the stairway-like pattern created by the points of $\cP$ 
arranged on the parabola is traversed alternately on one branch and then on the other (see the image on the right of Figure~\ref{FigureParabolas}).

\subsubsection{Chords in the parabolic-taxicab distance}
In distance $\dP$, the chords and, in particular, the diameters (the longest chords) 
of $\partial\scBpc(O,r)$, have unusual characteristics compared 
to those in Euclidean spaces.
For example, if $r=6$, there are no chords of length $2\cdot 6=12$, 
so that the parabolic-taxicab diameter of $\partial\scBpc(O,6)$ is $10$.
Also, we note that in this case, $P=(-6,-6)$ has four diametrically 
opposite points, while $P=(2,8)$ and $P=(21,15)$ have none.
\begin{figure}[htb]
 \centering
 \hfill
    \includegraphics[height=\textwidth,angle=-90]{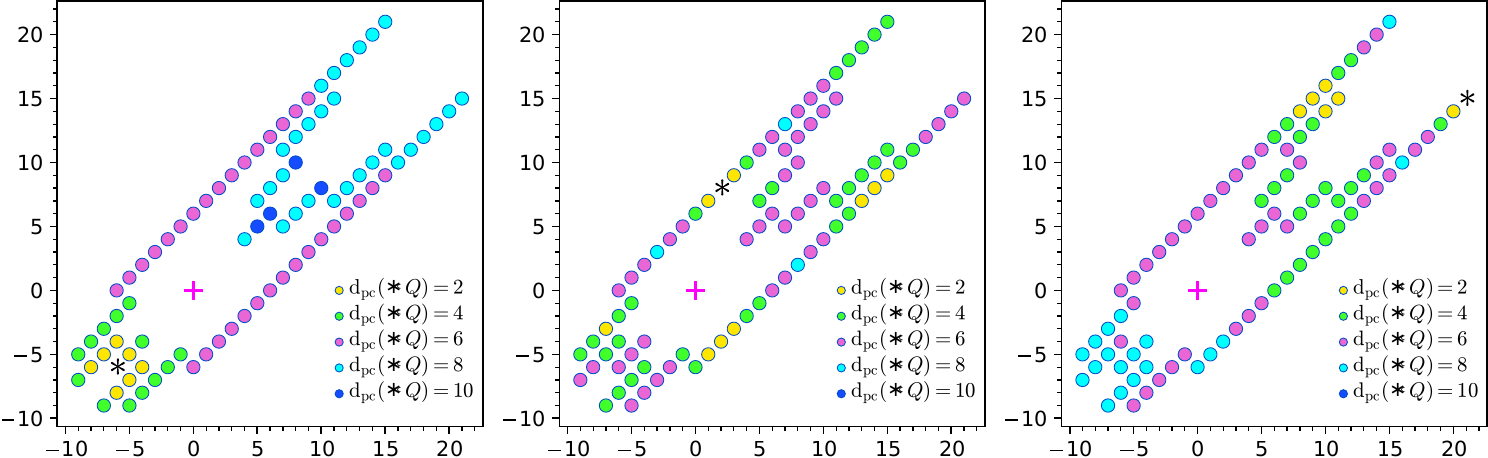}
    \hfill\mbox{}\\
\caption{The boundary of $\scBpc(O,6)$. 
Three points $P$ marked with $\ast$
are chosen as follows: 
$P=(-6,-6)$ (left), 
$P=(2,8)$ (middle), 
$P=(21,15)$ (right). 
All points $Q\in\partial\scBpc(O,6)$ are shown in colors indicating 
the parabolic-taxicab distance from $P$ to~$Q$.
}
\label{FigureChords}
 \end{figure}
Furthermore, we found that all the tested chords have an even 
parabolic-taxicab length as we did not find any with an odd length.
In Figure~\ref{FigureChords}, all points situated at a certain 
parabolic-taxicab distance from the three chosen points are shown
in a specific color and their frequencies are given in Table~\ref{TableChords}.
\begin{center}
\addtolength{\tabcolsep}{-1pt}
\begin{table}[ht]
\centering
\scalebox{0.92}{ 
\begin{tabular}{lccccc}
   \toprule
   & $\dP(P,Q)=2$ & $\dP(P,Q)=4$ & $\dP(P,Q)=6$ & $\dP(P,Q)=8$ & $\dP(P,Q)=10$
     \\ \midrule
$P=(-6,-6)$ & 7 & 13 & 32 & 29 & 4 \\[1mm]
$P=(2,8)$ & 6 & 23 & 37 & 19 & 0 \\[1mm]
$P=(21,15)$ & 9 & 32 & 41 & 3 & 0 \\
  \bottomrule
\end{tabular}
}  
\caption{The frequency of points $Q\in\partial\scBpc(O,6)$
that are at the parabolic-taxicab distance $2,4,6,8,10$ from $P=(-6,-6)$, 
$P=(2,8)$,  and $P=(21,15)$. 
The location of these points is indicated by colors in Figure~\ref{FigureChords}.
(Note, the cardinality $\#\partial\scBpc(O,6)=86$.) 
}
\label{TableChords}
\end{table}
\end{center}
A more in-depth study is necessary to characterize the properties of the
parabolic-taxicab distances between pairs of points in $\partial\scBpc(O,r)$.
Here are some interesting questions to consider.
Here are some interesting questions to consider.
\begin{problem}\label{Problem}
    Let $r>3$ be an integer.
 \begin{enumerate}[label=\textup{(\roman*)}]
\item\label{Problem1}
What is the parabolic-taxicab diameter of $\partial\scBpc(O,r)$?
\item\label{Problem2}
Show that there are no chords $PQ$ with $P,Q\in\partial\scBpc(O,r)$
and odd $\dP(P,Q)$.
\item\label{Problem3}
Is it true that the parabolic-taxicab diameter of $\partial\scBpc(O,r)$
is $\leq 2r$?
\end{enumerate}
\end{problem}

Short arguments proving part~\ref{Problem2} and part~\ref{Problem3} above (the latter in the affirmative) were communicated to the authors by Omer Cantor. We present them below.

For~\ref{Problem2}, consider the graph $G$ on $\ZZ^{2}$ in which two distinct 
vertices $(a_1,b_1)$ and $(a_2,b_2)$ 
are adjacent if one can reach 
one of the points from the other by a single application 
of $L', L'', M', M'', M'''$, or $M^{iv}$. More explicitly, the edges of $G$ are of the form $\left\{(m,k),(n,k)\right\}$ 
or $\left\{(k,m),(k,n)\right\}$, where 
$m+n = 2k+1$ or (without loss of generality,
assuming $m \geq n$) $m-n = 1$. 
Then, for each edge in $G$, we have $m \not\equiv n \pmod 2$ when represented in the aforementioned form. In particular, if two vertices $(a_1,b_1)$ and $(a_2,b_2)$ are adjacent, then $(a_1+b_1) \not\equiv (a_2+b_2) \pmod 2$. Thus, any closed walk in $G$ has even length, and so, for points $A,B,C \in \ZZ^2$, the sum $\dP(A,B)+\dP(B,C)+\dP(A,C)$ is even. It follows that if $A,B \in \partial\scBpc(C,r)$, then $\dP(A,B)$ is even.

The conclusion that the answer to~\ref{Problem3}
is ``yes'' follows from the fact that the function $f(r) = 2r - \operatorname{diam}\left(\partial\scBpc(v, r)\right)$ is non-decreasing in $r$. This can be seen, for instance, by observing that if $P,Q \in \partial\scBpc (O,r+1)$ are distinct and $P',Q' \in V(G)$ with $P'$ (resp. $Q'$) adjacent to $P$ (resp. $Q$), then $\dP(P,Q) \leq \dP(P',Q')+2 \leq \operatorname{diam}(\partial\scBpc(O,r))+2$, whence 
$\operatorname{diam}\left(\partial\scBpc(O,r+1)\right) \leq \operatorname{diam}\left(\partial\scBpc(O,r)\right)+2$. Then, since $\partial\scBpc(O,1) \leq 2$, 
it follows $\partial\scBpc(O,r) \leq 2r$ for $r \geq 1$.

\subsubsection{The Euclidean width and diameter} 

The geometric form of the ball is somewhat reminiscent of the shape of the swallows
representing the neighbor spacing distribution between Farey fractions 
(see~\cite{ABCZ2001,CZ2003, CVZ2010}), only this time the wings are attached 
along the body.
According to Theorem~\ref{Theorem0}, for each $r\ge 0$, 
the ball $\scBpc(O,r)$ is tangent to and completely 
contained between the lines $y=x\pm r$. Thus, the \textit{Euclidean width} of $\scBpc(O,r)$ is $\sqrt{2}r$.

The \textit{Euclidean diameter} of $\scBpc(O,r)$ is determined by 
the pairs of points farthest away from each other (in the Euclidean distance).
Theorem~\ref{Theorem0} shows that there are two such pairs of points, symmetric to each other with respect to the first diagonal $y=x$.
Indeed, by formula~\eqref{eqChar3}
we can choose $Q=Q(r)$, where
\begin{align}\label{eqQr}
   Q(r)=\big(r(r+1)/2,r(r-1)/2\big), \ \ \text{ for $r\ge 0$}.
\end{align}

Letting $\dE$ denote the Euclidean distance,
we find that for any $(x,y)\in \Qt$, we have:
\begin{align}\label{eqdE1}
  \dE\big(Q,(x,y)\big) &> \dE\big(Q,(x+1,y-1)\big),\\
  \dE\big(Q,(x,y)\big) &> \dE\big(Q,(x,y+1)\big)\label{eqdE2}.
\end{align}
Then,~\eqref{eqdE1} shows that $P$, the Euclidean diametrically opposite point to $Q$, 
must lay above the diagonal $x=y$ and 
as far to the left as possible, 
and condition~\eqref{eqdE2} shows that $P$ must be 
the lowest point among the potential multiple options of points that could be located on the same leftmost vertical line that borders the ball.

We let $P=P(r)$, where
\begin{align}\label{eqPr}
   P(r)=\big(x(r),y(r)\big),
\end{align} 
be the Euclidean diametrically opposite point to $Q(r)$. 
Let $c=c(r)$ be the $y$-intercept of the line 
$y=x+c$ where point $P(r)$ lies.
Then, following~\eqref{eqChar2}, we find that, 
for $r\ge 0$, we have:
\begin{align}
  c(r) & = r - 2\lfloor{r/3}\rfloor;\label{eqcr}\\
  x(r) & = -\lfloor{(r+1)(r+2)/6}\rfloor;\label{eqxr}\\
  y(r) & = x(r)+c(r).\label{eqyr}
\end{align}
Note that sequence $\{c(r)\}_{r\ge 0}$, 
whose elements begin with:
$0,1,2,1, 2, 3, 2, 3, 4, 3, 4, 5, 4, \dots$, 
follows the `terza rima' rhyme scheme of Dante Alighieri's 
``\textit{Divine Comedy}'',
starting from its second term~\mbox{\cite[{\href{https://oeis.org/A008611}{A008611}}]{oeis}}.
The absolute values of the initial elements of $\{x(r)\}_{r\ge 0}$ are:
$0,1,2, 3,5,7,9,12,15,18,\dots$, 
and the absolute values of the initial elements of $\{y(r)\}_{r\ge 0}$ are:
$0,0,0,  2,3,4,7,9,11,15,\dots$
(see \cite[{\href{https://oeis.org/A001840}{A001840}}]{oeis}, 
\cite[{\href{https://oeis.org/A236337}{A236337}}]{oeis}
and the references cited there for an overview of the properties that these sequences have).

Let $\diamE\big(\scBpc(O,r)\big)$ denote the Euclidean diameter of the ball.
Then, on using relations~\eqref{eqQr}, \eqref{eqPr},
\eqref{eqcr},\eqref{eqxr}, and \eqref{eqyr}, we 
calculate $\dE(P,Q)$, which gives the 
following closed form formula for the Euclidean diameter:
\begin{align*}
  \diamE&\big(\scBpc(O,r)\big) \\
  &= \Bigg(
   \bigg(\sdfrac{r(r+1)}{2} 
      + \Big\lfloor\sdfrac{(r+1)(r+2)}{6}\Big\rfloor\bigg)^2
   + \bigg(\sdfrac{r(r-3)}{2} 
      + \Big\lfloor\sdfrac{(r+1)(r+2)}{6}\Big\rfloor
      + 2\Big\lfloor\sdfrac{r}{3}\Big\rfloor\bigg)^2
  \Bigg)^{1/2},
\end{align*}
for $r\ge 0$. This generates 
a sequence in which the squares of the initial elements are:
$0,4,26, 106, 306, 680, 1360$, $2500, 4122, 6516,\dots$

\subsubsection{Remarks on the symmetry}
Let `$\sim$' be the operator that makes the associations:
$\widetilde{L'}=L''$; $\widetilde{L''}=L'$;
$\widetilde{M'}=M'''$; $\widetilde{M'''}=M'$;
$\widetilde{M''}=M^{iv}$; $\widetilde{M^{iv}}=M''$.
This pairwise symmetry of the operators in $\cL\cup\cM=\{L',L'',M',M'',M''',M^{iv}\}$ 
causes, geometrically, the points transformed by their compositions 
to mirror across the first diagonal.
To be exact, if $F_1,\dots,F_n\in\cL\cup\cM$, then for every $(a,b)\in\ZZ^2$
the midpoint of the segment with endpoints 
$F_1\circ\dots\circ F_n(a,b)$ and
$\widetilde{F_1}\circ\dots\circ \widetilde{F_n}(a,b)$ has equal coordinates.
This explains why $\scBpc(O,r)$ is symmetric with respect to $y=x$ for all radii $r\ge 0$.
Likewise, the more general balls whose centers are sets of lattice points that are
symmetric with respect to the first diagonal
maintain their symmetry for all $r\ge 0$ (see Figure~\ref{FigureSymmetries} 
for two such examples).
\begin{figure}[htb]
 \centering
 \hfill
    \includegraphics[height=\textwidth,angle=-90]{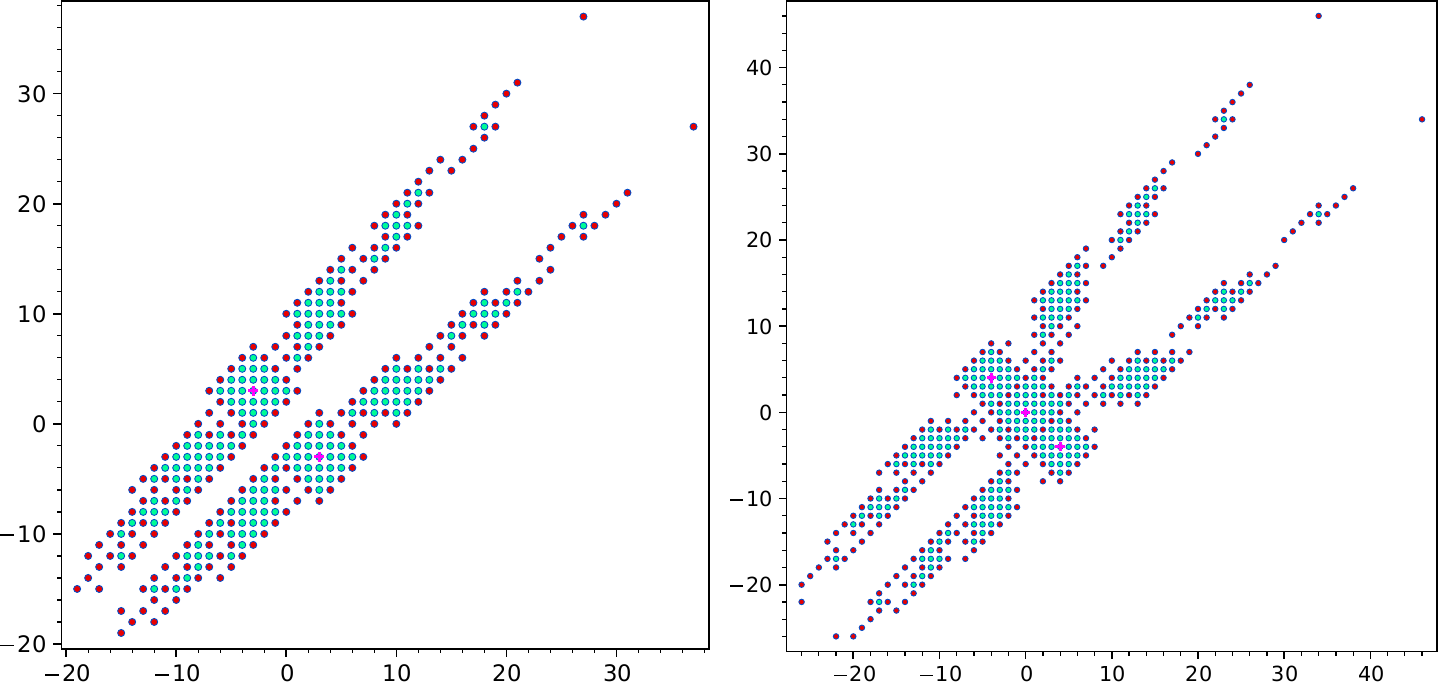}
    \hfill\mbox{}\\
\caption{The set of lattice points situated at distance at most $4$ from the points in the set
$\cC = \big\{(3,-3), (-3,3)\big\}$ (left)
and
$\cC = \big\{(0,0),(4,-4), (-4,4)\big\}$ (right). 
}
\label{FigureSymmetries}
 \end{figure}

However, the projections onto the first coordinate of 
the cross-sections made in $\scBpc(O,r)$ are not identical 
for $c<0$ and $c>0$, and although the formulas
for $\Sa(c,r)$ and $\Sm(c,r)$ in the two cases 
can be derived from one another, 
we have chosen to keep them all in \cref{Theorem0} for completeness.
In the proofs that follow in the next sections, we present the arguments for one of the cases $c \geq 0$ or $c \leq 0$, or just for a specific value of $c$, with the goal of concluding a corresponding result for the value $-c$, too. 
Then, symmetry, or the same framework of proof described, can be used to deduce the given result for all $c$ under consideration.

\section{Preparatory lemmas and the proof of Theorem~\ref{Theorem0} Part~\ref{Theorem0:I}
and Part~\ref{Theorem0:II}\ref{Theorem0:IId}}
Our first result shows that the ball 
$\scBpc(O,r)$
is contained between the lines with equations $y = x+|r|$, for every $r\ge 0$, as stated in~\cref{Theorem0}.
\label{preplemmas}
\begin{lemma}\label{lemma1}
Let $r\ge 0$ be an integer.
    If $|c| > r$, then $\cF(r,c) = \emptyset$. 
\end{lemma}
\begin{proof}
The statement verifies if $r = 0$ because $\scBpc(O,0)= \{(0,0)\}$. 
If $r=1$, then \mbox{$L'(0,0) = (1,0)$}, and $L''(0,0) = (0,1)$, 
as seen in Figure~\ref{FigureSevenBalls},
and the rest of the points at parabolic-taxicab distance $1$ from the origin are $(-1,0)$ and $(0,-1)$, so the statement is also verified.

Let us assume the statement holds for all radii $r$, $1 \leq r \leq R$. Suppose there exists a point \mbox{$(a,b) \in \partial \scBpc(O,R+1)$} with $b = a+c$ for some $|c| > R + 1$. 
Then, since both \mbox{$|c| + 1 > R$}, and $|c| - 1 > R$,
and the points 
at taxicab distance $1$ from $(a,b)$ 
lie on the lines \mbox{$y = x + (|c| \pm 1)$}, by induction there must exist an integer $a^{*}$ such that $(a^{*},b) \in \partial \scBpc(O,R)$ with 
\mbox{$L'(a^{*},b) = (a,b)$}, or an integer $b^{*}$ such that $(a,b^{*}) \in \partial \scBpc(O,R)$ with $L''(a,b^{*}) = (a,b)$. 

But if the former were true, then $b - a^{*} = a - b - 1 = -(c+1)$. Then, if $c \geq 0$, we have $|b - a^{*}| > R+2$, and if $c < 0$, then $|b - a^{*}| > R$, contradicting the inductive assumption. Similarly, if the latter were true, we obtain $b^{*} - a = a - b + 1 = -(c-1)$ and so $|b^{*} - a| > R$, which is once again a contradiction, thus concluding the proof of Lemma~\ref{lemma1}.
\end{proof}

\begin{figure}[htb]
 \centering
 \hfill
     \includegraphics[height=\textwidth,angle=-90]{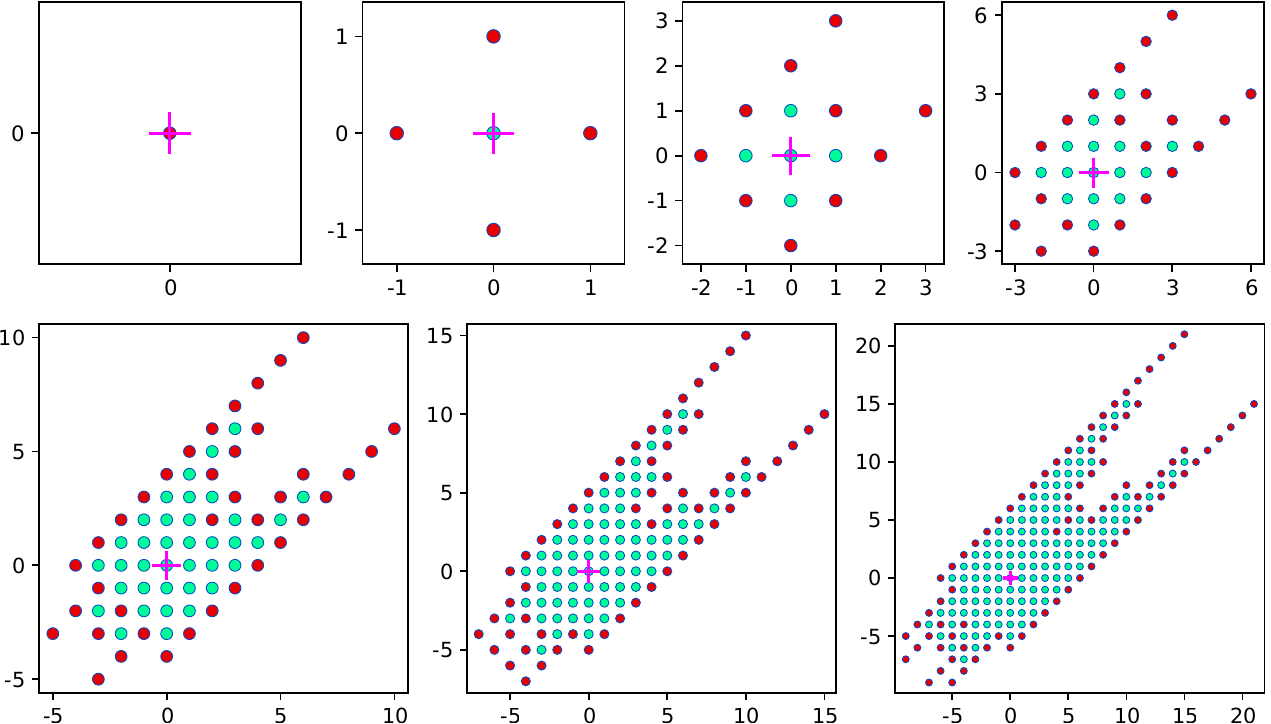}
    \hfill\mbox{}\\	 
\caption{The smallest seven balls with highlighted centers and boundaries.
Notice the developing pulsatory effect that is becoming visible here, mainly along the diagonal $y=x$ as the radius increases.}
\label{FigureSevenBalls}
 \end{figure}
\begin{lemma}\label{lemma2}
Let $r\ge 0$ be an integer.
If $|c| < r$ and $c \not\equiv r \pmod 2$, then $\cF(r,c) = \emptyset$.
\end{lemma}

\begin{proof}
This is vacuous for $r = 0$, and evident for $r = 1$ since the points 
in the ball of radius~$1$ centered at $O$ in the parabolic-taxicab distance lie on the lines $y = x \pm 1$. Suppose the statement holds for $1 \leq r \leq R$. If there is a point $(a,b) \in \cF(R+1,c)$ such that $|c| < R + 1$ and $c \not\equiv R + 1 \pmod 2$, then since the four points at taxicab distance $1$ from $(a,b)$ lie on the lines $y = x + (c \pm 1)$ and $c \pm 1 \not\equiv R \pmod 2$, these points cannot be members of $\partial \scBpc(O,R)$ by induction. 
    
Thus, $L''(a,b^{*}) = (a,b)$ for some $(a,b^{*}) \in \partial \scBpc(O, R)$, or $L'(a^{*},b) = (a,b)$ for some $(a^{*},b) \in \partial \scBpc(O,R)$. 

In the former case, we have $b^{*}-a = a-b +1 = -(c-1) \not\equiv c \pmod 2$. However, this is precluded by induction if $|-(c-1)| < R$, since $c \equiv R \pmod 2$. If $|-(c-1)| \geq R$, since $\cF(R,c') = \emptyset$ for $|c'| > R$, we have $|-(c-1)| = R$, contradicting $c \not\equiv R + 1 \pmod 2$. 

Similarly, in the latter case, we have $b-a^{*} = a-b-1 = -(c+1) \not\equiv c \pmod 2$. This is impossible if $|-(c+1)| < R$ by induction, since $c \equiv R \pmod 2$. Otherwise, we obtain $|-(c+1)| = R$, 
which contradicts $c \not\equiv R + 1 \pmod 2$. 

This concludes the proof of Lemma~\ref{lemma2}.
\end{proof}
Note that ~\cref{lemma1,lemma2} prove ~\cref{Theorem0} Part~\ref{Theorem0:I}

\begin{lemma}\label{lemma3}
For any integer $r\ge 0$, we have:
\begin{align*}
   \cF(r,r) \setminus \Qu  &= \big(\scL(r) \cap \Qd\big) \setminus \{(0,r)\},\\
\intertext{and}
   \cF(r,-r) \setminus \Qu &= \big(\scL(-r) \cap \Qp\big) 
             \setminus \{(0,-r)\}.
\end{align*}
\end{lemma}

\begin{proof}
 When $r = 0$ this is clear since $\cF(0,0) \setminus \Qu = \emptyset$ 
 and the only point in $\scL(0)$ that lies in $\Qd$ (resp. $\Qp$) is $O$. 
 For $|r| = 1$, the statements are true since $\partial \scBpc(O,1) = \{(0,\pm 1), (\pm 1, 0)\}$. 
 Suppose the the statements hold for $1 \leq |r| \leq R$. 

Each point $(a,a+R)$ in $\cF(R,R) \setminus \Qu$ contributes a unique point $(a-1,a+R)$ to $\cF(R+1,R+1) \setminus \Qu$. Moreover, $(a,a+R+1) = ((a+1)-1,a+R+1)$ and the point $(a+1,a+R+1) \in \cF(R,R) \setminus \Qu$ as long as $a < -1$. Since $(0,R) \in \Qu$, we obtain that $(-1,R-1)$ is the only point in $\cF(R,R) \setminus \Qu$ that contributes two points to $\cF(R+1,R+1) \setminus \Qu$, namely $(-2,R-1)$ and $(-1,R)$.

Thus, 
\begin{align*}
   |\cF(R+1,R+1) \setminus \Qu| \geq |\cF(R,R) \setminus \Qu| + 1. 
\end{align*}
Note that
\begin{align*}
    |\scL(R+1) \cap \Qd|-1 
    = |\scL(R) \cap \Qd| = |\cF(R,R) \setminus \Qu| + 1,
\end{align*}
by induction, and each contributed point above is a member of $\Qd$, lies on the line \mbox{$y = x + (R + 1)$}, and has strictly negative first coordinate (that is, each such point is in the set $(\scL(R+1) \cap \Qd) \setminus \{(0,R+1)\}$). 

We conclude that 
\begin{align*}
    \cF(R+1,R+1) \setminus \Qu \supseteq \scL(R+1) \cap \Qd \setminus \{(0,R+1)\}.
\end{align*}
If there was another point $(a,b)$ in $\cF(R+1,R+1) \setminus \Qu$, it follows that $a < -(R+1)$ and $b < 0$, and so by induction $(a+1,b), (a,b-1) \not\in \cF(R,R)$. 
Since $\cF(R,c) = \emptyset$ for $|c| > R$, the other two points at taxicab distance $1$ from $(a,b)$ are not in $\partial \scBpc(O,R)$, so it must be the case that $(a,b) = L'(a^{*},b)$ for some $a^{*} \in \ZZ$ such that $(a^{*},b) \in \partial \scBpc(O,R)$, or $(a,b) = L''(a, b^{*})$ for some $b^{*} \in \ZZ$ such that $(a,b^{*}) \in \partial \scBpc(O,R)$. But then either $a^{*} = 2b - a + 1$ or $b^{*} = 2a - b + 1$. 

In the former case, we have $b - a^{*} = a - b - 1 = -(R + 1) - 1 = -R - 2$, which is impossible since $\cF(R,-R-2) = \emptyset$. 

In the latter case, $b^{*} - a = a - b + 1 = -(R+1) + 1 = -R$: this is also impossible since induction then yields that $(a,b^{*}) \in \Qp$. 

By symmetry or adapting the above argument from this proof,
we also obtain
\begin{equation*}
  \cF(R+1,-(R+1)) \setminus \Qu 
  = \big(\scL(-(R+1)) \cap \Qp\big) \setminus \{(0,-(R+1))\},
\end{equation*}    
thus concluding the proof of Lemma~\ref{lemma3}.
\end{proof}

\begin{lemma}\label{lemma4}
For every integer $r\ge 0$, we have:
\begin{align*}
  \max \Sa(r,r) &= \binom[0.81]{r}{2},\\
\intertext{and}
  \max \Sa(r,-r) &= \binom[0.81]{r+1}{2}.
\end{align*}
Moreover, the following statements hold:
\begin{enumerate}
  \item[\namedlabel{Lemma41}{\normalfont{(1)}}]
    If $(x,x+r) \in \scL(r) \cap \Qu$ and $x \leq \binom{r}{2}$,
            then $(x,x+r) \in \cF(r,r)$;
  \item[\namedlabel{Lemma42}{\normalfont{(2)}}]
    If $(x,x-r) \in \scL(-r) \cap \Qu$ and $x \leq \binom{r+1}{2}$, 
    then $(x,x-r) \in \cF(r,-r)$.
\end{enumerate}
\end{lemma}

\begin{proof}
The two claims in the first sentence of the lemma hold for $r \leq 1$. 
They also hold for~$r = 2$, since the ball of radius $2$ centered at $O$ in the parabolic-taxicab distance is the ball of radius $2$ in the taxicab distance along with the points $(1,3)$ and $(3,1)$,
as seen in Figure~\ref{FigureSevenBalls}. 

Suppose the two statements hold for all $2 \leq r \leq R$. 
Then
\begin{align*}
  \Big(\binom[0.69]{R+1}{2}, 
      \binom[0.69]{R+2}{2}\Big) \in \scL(R+1),\ 
    \text{ and, moreover,}\ 
  \Big(\binom[0.69]{R+1}{2}, 
      \binom[0.69]{R+2}{2}\Big) \in \cF(R+1,R+1),
\end{align*}
since, by induction,
\begin{align*}
  \Big(\binom[0.69]{R+1}{2},\binom[0.69]{R}{2}\Big) \in \cF(R,-R) ,
\end{align*}
and by definition of $L''$,
\begin{align*}
  L''\Big(\binom[0.69]{R+1}{2}, \binom[0.69]{R}{2}\Big)
  = \Big(\binom[0.69]{R+1}{2}, \binom[0.69]{R+2}{2}\Big).
\end{align*}

Suppose there is a point $(a,a+R+1) \in \cF(R+1,R+1)$ with $a > \binom{R+1}{2}$. Note that no point at taxicab distance $1$ from $(a,a+R+1)$ is in $\partial \scBpc(O, R)$ by induction and \cref{lemma1} (no point in $\partial \scBpc(O,R)$ lies on the line $y = x+R+1$). Since $L'$ and $L''$ are involutions, it follows that $L''(a,a+R+1) \in \partial \scBpc(O, R)$ or $L'(a,a+R+1) \in \partial \scBpc(O, R)$. However, both possibilities are precluded by induction and \cref{lemma1} (no point in $\partial \scBpc(O, R)$ lies on the line $y = x-R-1$). 

By a slight modification of the argument in the last paragraph, 
we can also show that there can be no point in $\cF(R+1,-(R+1))$ with first coordinate larger than $\binom{R+2}{2}$. 

We now prove statement~\ref{Lemma41} in the second half of the lemma. (Then,
the proof can be simply adapted to prove statement~\ref{Lemma42}, or
alternatively,~\ref{Lemma42} can be deduced by symmetry.)

The statement holds for $r \leq 1$, as well as for $r = 2$, by the explicit description of $\partial \scBpc(O,2)$ given above. 
Suppose the statement holds for all $2 \leq r \leq R$. 
Then, if \mbox{$(a,a+R+1) \in \scL(R+1)$} with $a \leq \binom{R+1}{2}$, we have $L''(a,a+R+1) \in \scL(-R)$. Since $L''$ preserves the first coordinate, induction yields that $L''(a,a+R+1) \in \cF(R,-R)$, whence $(a,a+R+1) \in \cF(R+1,R+1)$, because $L''$ is an involution. 

This concludes the proof of Lemma~\ref{lemma4}.
\end{proof}

By \cref{lemma3,lemma4}, \cref{Theorem0} Part~\ref{Theorem0:II}\ref{Theorem0:IId} follows.

\section{Proof of Theorem~\ref{Theorem0} Part ~\ref{Theorem0:II}\ref{Theorem0:IIa}
}
Since the arguments in the proof are similar for the $y$-intercept $c$ 
in the cases $c\geq 0$ and $c<0$, we will only analyze the former case. 
We also mention that the results can be obtained from one another, by symmetry.

The statement is vacuous for $r \leq 1$, 
and is directly verified for $r = 2$. 
Suppose it holds for $2 \leq r \leq R$. 
Note first that 
\begin{align*}
  \left[\binom[0.81]{R-1}{2}+1,\binom[0.81]{R}{2}+1\right] \cap \ZZ \subseteq \Sa(R+1,R-1) \cap \ZZp,  
\end{align*}
since, by \cref{lemma4,lemma1}, 
\begin{align*}
  \Sa(R-1,R-1) \cap \ZZp &= \left[0,\binom[0.81]{R-1}{2}\right] \cap \ZZ,\\
  \Sa(R,R) \cap \ZZp &= \left[0,\binom[0.81]{R}{2}\right] \cap \ZZ,\\
  \intertext{and}
 \Sa(R',R-1) &= \emptyset ,\ \text{ for $R' < R-1$.}
\end{align*}

If there was a point $(a,a+R-1) \in \cF(R+1,R-1)$ with $a > \binom{R}{2} + 1$, induction yields
\begin{align*}
  \Sa(R,R-2) \cap \ZZp = 
  \left[\binom[0.81]{R-2}{2}+1,\binom[0.81]{R-1}{2}+1\right] \cap \ZZ.  
\end{align*}
Therefore, it follows that $L'(a^{*},a+R-1) = (a,a+R-1)$ for some $a^{*} \in \ZZ$ such that $(a^{*},a+R-1) \in \partial \scBpc(O,R)$, or $L''(a,b^{*}) = (a,a+R-1)$ for some $b^{*} \in \ZZ$ such that $(a,b^{*}) \in \partial \scBpc(O,R)$. 

In the former case, we have $a = 2(a+R-1)-a^{*}+1$, and hence $(a+R-1)-a^{*} = a - (a + R - 1) - 1 = -R$. But then $a^{*} = a + 2R - 1 > \binom{R}{2} + 2R = \binom{R+1}{2} + R$, which contradicts the fact that 
\begin{align*}
  \Sa(R,-R) \cap \ZZp = \left[0,\binom[0.81]{R+1}{2}\right] \cap \ZZ . 
\end{align*}

In the latter case, we have $a + R - 1 = 2a - b^{*} + 1$, 
so that $b^{*} - a = -(R-2)$. But this is impossible since 
\begin{align*}
  \Sa(R,-(R-2)) \cap \ZZp 
  = \left[\binom[0.81]{R-1}{2}+1,\binom[0.81]{R}{2}\right] \cap \ZZ,
\end{align*}
by induction, provided $b^{*} \geq 0$. If $b^{*} < 0$, we obtain $\binom{R}{2} + 1 < a < R - 2$, which is a contradiction. 

The conclusion is that 
\begin{align*}
  \Sa(R+1,R-1) \cap \ZZp 
  = \left[\binom[0.81]{R-1}{2}+1,\binom[0.81]{R}{2}+1\right] \cap \ZZ.  
\end{align*}

Now, fix a non-negative $y$-intercept $0 \leq c < R - 1$ with $c \equiv R + 1 \pmod 2$, 
and write $c = R + 1 - 2k$. 
We can check that 
\begin{align*}
   \left[\binom[0.81]{R-k}{2} + k, \binom[0.81]{R+1-k}{2} + k\right] 
   \cap \ZZ \subseteq \Sa(R+1,c) \cap \ZZp, 
\end{align*}
since, by induction,
\begin{align*}
  \Sa(R-1,c) \cap \ZZp &= \left[\binom[0.81]{R-2-(k-1)}{2} + k-1,
     \binom[0.81]{R-1-(k-1)}{2} + k-1\right] \cap \ZZ,\\
\intertext{and}
  \Sa(R,c+1) \cap \ZZp& = \left[\binom[0.81]{R-1-(k-1)}{2}+k-1,
    \binom[0.81]{R-(k-1)}{2}+k-1\right].
\end{align*}
Note that induction also yields:
\begin{align*}
  &\quad\left[0, \binom[0.81]{R-k}{2} + k - 1\right] \cap \ZZ\\
  &= \left(\left[0,\binom[0.81]{R-2k+1}{2}\right] \cap \ZZ\right) 
  \sqcup
   \bigsqcup_{1 \leq j \leq k-1} 
    \left(\left[\binom[0.81]{R-k-j}{2} + k - j,
        \binom[0.81]{R-k-j+1}{2} + k - j\right] \cap \ZZ\right)\\
  &= \bigsqcup_{1 \leq j \leq k} 
  \big(\Sa(R+1-2j,c) \cap \ZZp\big).
\end{align*}

A consequence is that 
\begin{align*}
  \min\big( \Sa(R+1,c) \cap \ZZp\big) 
  = \binom[0.81]{R-k}{2} + k.  
\end{align*}
Suppose $(a, a + c) \in \cF(R+1,c)$ with $a > \binom{R+1-k}{2} + k$. We observe that
\begin{align*}
  \Sa(R,c-1) \cap \ZZp 
  = \left[\binom[0.81]{R-1-k}{2}+k,
       \binom[0.81]{R-k}{2}+k\right] \cap \ZZ ,\
       \text{ if $c \geq 1$, }
\end{align*}
and 
\begin{align*}
  \Sa(R,c-1) \cap \ZZp = 
    \left[\binom[0.81]{R-k}{2}+1,
       \binom[0.81]{R-k+1}{2}\right] \cap \ZZ ,\
       \text{ if $c \geq 0$, } 
\end{align*}
by induction. It therefore follows that $L'(a^{*},a+c) = (a,a+c)$ for some $a^{*} \in \ZZ$ such that $(a^{*},a+c) \in \partial \scBpc(O,R)$, or $L''(a,b^{*}) = (a,a+c)$ for some $b^{*}\in\ZZ$ such that 
\mbox{$(a,b^{*}) \in \partial \scBpc(O,R)$}. 

In the former case, we have $2(a+c)-a^{*}+1 = a$, so that $(a+c)-a^{*} = -(c+1)$. But then $a^{*} = a+2c+1 > \binom{R+1-k}{2}+k+2c+1$. By induction, 
\begin{align*}
  \Sa(R,-(c+1)) \cap \ZZp 
  = \left[\binom[0.81]{R-k}{2}+1,
           \binom[0.81]{R-k+1}{2}\right] \cap \ZZ,
\end{align*}
and so we arrive at a contradiction.

In the latter case, we have $2a-b^{*}+1=a+c$,
so that $b^{*}-a = -(c-1)$. But, by induction, 
\begin{align*}
  \Sa\big(R,-(c-1)\big) \cap \ZZp 
     = \left[\binom[0.81]{R-1-k}{2}+k,
        \binom[0.81]{R-k}{2}+k\right] \cap \ZZ,\
       \text{ if $c \geq 1$,  }
\end{align*}
and
\begin{align*}
  \Sa\big(R,-(c-1)\big) \cap \ZZp 
   = \left[\binom[0.81]{R-k}{2}+k-1,
       \binom[0.81]{R-k+1}{2}+k-1\right] \cap \ZZ,\
       \text{ if $c =0$,  }
\end{align*}
both of which are precluded by the inequality $a > \binom{R+1-k}{2} + k$, if $b^{*} \geq 0$. If $b^{*} < 0$, we have that $\binom{R+1-k}{2}+k < a < R-2k$, which contradicts the assumption $k \geq 2$. 

\section{Proof of Theorem~\ref{Theorem0} Parts~\ref{Theorem0:II}\ref{Theorem0:IIb} and \ref{Theorem0:IIc}
}
\subsection{A convenient signed counter function}
For every integer $x \in \ZZ$, we define \begin{equation}\label{epsilonandz}
  \z(x) := \epsilon(x)\big(|x|+1\big),\ \text{ where }\
    \epsilon(x) 
    := \begin{cases}
        -1, &\text{ if } x > 0; \\[4pt]
        1, &\text{ if } x \leq 0.
  \end{cases}
\end{equation}
Let $\z^{(k)}$ denote the $k$-fold composition of $\z$ with itself, 
that is, 
$\z^{(k)} := \underbrace{\z \circ \cdots \circ \z}_{k \text{ times}}$. 

\begin{remark}\label{Remarkz}
 We list some basic properties of $\z$. 
\begin{enumerate}[wide, labelwidth=0pt]
\setlength{\itemsep}{0pt}
    \setlength{\parskip}{8pt}
    \setlength{\parsep}{0pt} 
    \item 
By definition~\eqref{epsilonandz}, $\z(x)\neq 0$ 
for every $x \in \ZZ$, while 
$\z(x) > 0$, if $x \leq 0$, and
$\z(x) < 0$, if $x > 0$.
\item 
For every integer $k\ge 0$, we have
\begin{align}\label{eqzk}
  \z^{(k)}(x) = 
   \begin{cases}
    (-1)^{k}(x+k), & \text{ if $x > 0$; }  \\[8pt]
    (-1)^{k-1}(-x+k),  & \text{ if $x \le 0$. }
   \end{cases}
\end{align}
\begin{proof}
    
If $y \in \ZZ$, note that $(y-|x|) - (y-1-|\z(x)|) = -|x|+1+|x|+1 = 2$. 
Then, it follows that $\z^{(k)}(x) = (-1)^{k}(x+k)$ for every integer 
$k \geq 1$, provided $x > 0$. 
Indeed, for $k = 1$ this follows by definition of $\z$, and assuming 
the equality holds for $1 \leq k \leq K$, we deduce: 
\begin{align*}
  \z^{(K+1)}(x) = \z(\z^{(K)}(x)) 
  = \z\big((-1)^{K}(x+K)\big) = (-1)^{K+1}(x+K+1)  .
\end{align*}
\item
Similarly, when $x \leq 0$, we have $\z^{(k)}(x) = (-1)^{k-1}(-x+k)$.
Indeed, the equation holds in the case $k=1$, and
with the assumption that it 
holds for $1 \leq k \leq K$, it follows that
\begin{align*}
   \z^{(K+1)}(x) = \z\big((-1)^{K-1}(-x+K)\big) = (-1)^{K}(-x+K+1) ,
\end{align*}
which concludes the proof of~\eqref{eqzk}.
\end{proof}
\end{enumerate}

\end{remark}

\subsection{Recursion}\label{recursive}
The sets $\Sm(r,c)$ exhibit the following recursive behavior:
\begin{lemma}\label{lemma7}
Let $r\ge 0$ be an integer. 
Then, for every integer $c$ such that $|c| \leq r - 2$, and $c \equiv r \pmod 2$, we have:
\begin{align*}
\label{recformula}
   \Sm(r,c)  
   = \begin{cases}
     \Sm\big(r-1,\z(c)\big), &\text{ if } c \leq 0;\\[8pt]
     \Sm\big(r-1,\z(c)\big) + \z(c) - c, &\text{ if } c > 0.
   \end{cases}
\end{align*}
\end{lemma}
We will prove \cref{lemma7} in \cref{recproof}. Before proving \cref{lemma7}, we establish a useful lemma in the following section, which will condense the inductive argument that we utilize in \cref{recproof}. Informally, assuming that \cref{lemma7} holds for all integers $r \geq 0$ up to a non-negative integer $R$, we may ``unfold'' repeatedly using this assumption to obtain exact formulas for the sets $\Sm(r,c)$, as long as $r$ is a non-negative integer not exceeding $R$. 

\subsection{Exact Formulas}
We prove the following lemma:
\begin{lemma}\label{characterization}
Fix an integer $R \geq 0$. Assume, for every pair of integers $(r,c)$ such that $0 \leq r \leq R$, $|c| \leq r-2$, and $c \equiv r \pmod 2$, that 
\begin{align*}
   \Sm(r,c)  
   = \begin{cases}
     \Sm\big(r-1,\z(c)\big), &\text{ if } c \leq 0;\\[8pt]
     \Sm\big(r-1,\z(c)\big) + \z(c) - c, &\text{ if } c > 0.
   \end{cases}
\end{align*}
Then, if $c \leq 0$, we have: 
\begin{align*}
  \Sm(r,c) 
  = \begin{cases}
    \big([k-r,-1] \cap \ZZ\big) + c(k-1) - \binom[0.81]{k}{2},  
        &\text{ if } k \equiv 1 \pmod 2;\\[8pt]
    \big([0,r-k-1]\cap \ZZ\big) + ck - \binom[0.81]{k+1}{2}, 
        &\text{ if } k \equiv 0 \pmod 2,
  \end{cases}
\end{align*}
       and, if $c > 0$, we have:
\begin{align*}
 \Sm(r,c) 
 = \begin{cases}
    \big([0,r-k-1] \cap \ZZ\big) - c(k+1) - \binom[0.81]{k+1}{2},  
         &\text{ if } k \equiv 1 \pmod 2;\\[8pt]
    \big([k-r,-1] \cap \ZZ\big) - ck - \binom[0.81]{k}{2},
         &\text{ if } k \equiv 0 \pmod 2.
   \end{cases}
\end{align*}
\end{lemma}

\begin{proof}
First, suppose $c > 0$. 
Observe that 
\begin{align*}
  r'-k - \big|\z^{(k)}(c)\big| = r'-k-(r'-2k+k) = 0.  
\end{align*}
If $k \equiv 1 \pmod 2$, then by \cref{lemma3}, we have
$\Sm\big(r'-k,\z^{(k)}(c)\big) = [0,r'-k-1] \cap \ZZ$, 
and we obtain the following finite recursion by assumption:
\begin{align*}
  \Sm(r',c)  &= \Sm\big(r'-1,\z(c)\big) + \z(c) - c \\
        & = \Sm\big(r'-2,\z^{(2)}(c)\big)+\z(c)-c\\
        &= \cdots \\[-10pt]
        & = \Sm\big(r'-k,\z^{(k)}(c)\big) 
        + \bigg(\sum_{j=1}^{k}(-1)^{j+1}\z^{(j)}(c)\bigg)-c \\[-5pt]
        & = \big([0,r'-k-1] \cap \ZZ\big) 
          - c(k+1) - \binom[0.81]{k+1}{2}.
\end{align*}
If $k \equiv 0 \pmod 2$, then \cref{lemma3} yields 
    $\Sm\big(r'-k,\z^{(k)}(c)\big) = [k-r',-1] \cap \ZZ$, and 
\begin{align*}
    &\Sm(r',c) = \Sm\big(r'-1,\z(c)\big) + \z(c) - c \\
    & = \Sm\big(r'-2,\z^{(2)}(c)\big)+\z(c)-c\\
    & = \cdots \\[-10pt]
    & = \Sm\big(r'-k,\z^{(k)}(c)\big) 
     + \bigg(\sum_{j=1}^{k-1}(-1)^{j+1}\z^{(j)}(c)\bigg)-c \\[-5pt]
    & = \big([k-r',-1]\cap \ZZ\big)-ck-\binom[0.81]{k}{2},
\end{align*}
by assumption.

The case $c \leq 0$ is similar, where we once again repeatedly apply the assumption made in the statement of \cref{characterization}. Indeed, if $k \equiv 1 \pmod 2$, 
then by \cref{lemma3} we obtain 
$\Sm\big(r'-k,\z^{(k)}(c)\big) = [k-r',-1] \cap \ZZ$, and
\begin{align*}
    \Sm(r',c) &= \Sm\big(r'-1,\z(c)\big) \\
    & = \Sm\big(r'-2,\z^{(2)}(c)\big)+\z^{(2)}(c)-\z(c)\\
    & =\cdots \\[-10pt]
    & = \Sm\big(r'-k,\z^{(k)}(c)\big) 
       + \sum_{j=1}^{k-1}(-1)^{j}\z^{(j)}(c)\\[-5pt]
    & = \big([k-r',-1] \cap \ZZ\big) + c(k-1)-\binom[0.81]{k}{2}.
\end{align*}
If $k \equiv 0 \pmod 2$, then by \cref{lemma3}, we have
$\Sm\big(r'-k,\z^{(k)}(c)\big) = [0,r'-k-1]\cap\ZZ$, and
\begin{align*}
   \Sm(r',c) &= \Sm\big(r'-1,\z(c)\big)\\ 
    & = \Sm\big(r'-2,\z^{(2)}(c)\big)+\z^{(2)}(c)-\z(c)=\\
    & = \cdots \\[-10pt]
    & = \Sm\big(r'-k,\z^{(k)}(c)\big) 
        + \sum_{j=1}^{k}(-1)^{j}\z^{(j)}(c) \\[-5pt]
    & = \big([0,r'-k-1]\cap\ZZ\big) + ck - \binom[0.81]{k+1}{2}.
\end{align*}
This completes the proof of \cref{characterization}.
\end{proof} 

\subsection{Proof of Lemma~\ref{lemma7}  and the conclusion of the proof of Theorem~\ref{Theorem0}
}\label{recproof}
\begin{proof}[Proof of \texorpdfstring{\cref{lemma7}}{Lemma 5.1}]
For $r \leq 1$ the statement holds vacuously, and by the first part of the proof of \cref{lemma4}, $\Sm(2,0) = \{-1\} = \Sm(1,1)$. Thus, we may henceforth assume $r \geq 3$. 

Suppose the statement holds for all integers $r'$ such that $2 \leq r' \leq r-1$. We will first consider the case $c = \pm(r-2)$. Note that for each point $(a,a-r+1) \in \cF(r-1,-(r-1)) \setminus \Qu$, we have $L'(a,a-r+1) = (a-2r+3,a-r+1)$, and $a-r+1-(a-2r+3) = r-2$. Thus, $L'(x,x-r+1) \in \cF(r,r-2) \setminus \Qu$ and therefore 
\begin{align*}
  &\ \quad \Sm\big(r-1,\z(r-2)\big) + \z(r-2)-(r-2)\\ 
  &= \Sm\big(r-1,-(r-1)\big) - (r-1)-(r-2)\\
  &= ([0,r-2] \cap \ZZ)-(2r-3)\\ 
  &= \{a-2r+3 : (a,a-r+1) \in \cF(r-1,-(r-1)) \setminus \Qu\}\\ 
  &\subseteq \{a^{*} \in \ZZ : (a^{*},a^{*}+r-2) 
         \in \cF(r,r-2) \setminus \Qu\}\\ 
  &= \Sm(r,r-2),
\end{align*}
by \cref{lemma3}. 

If there was a point $(a,a+r-2) \in \cF(r,r-2) \setminus \Qu$ 
with $a \not\in [3-2r,1-r] \cap \ZZ$, 
since $1-r+(r-2) = -1$, by \cref{lemma3} it follows that $a < 3 - 2r$.
Thus, by \cref{lemma3} and the inductive assumption, there must exist 
$a^{*} \in \ZZ$ such that $(a^{*},a+r-2) \in \partial \scBpc(O,r-1)$ 
with $L'(a^{*},a+r-2) = (a,a+r-2)$, or $b^{*} \in \ZZ$ such that 
$(a,b^{*}) \in \partial \scBpc(O,r-1)$ with $L''(a,b^{*}) = (a,a+r-2)$.

In the former of these two possibilities, it follows that $a^{*} = a+2r-3$. This implies $(a^{*}, a+r-2) \in \cF(r-1,-(r-1))$, contradicting \cref{lemma3}. 

The latter possibility implies that $b^{*} = a-r+3$, and therefore $(a,b^{*}) \in \cF\big(r-1,-(r-3)\big)$. But then $2-r \leq a < 3-2r$ by induction, a contradiction.
    
For each $(a,a+r-1) \in \cF(r-1,r-1) \setminus \Qu$, we have 
$L''(a,a+r-1) = (a,a-r+2)$, and $a-r+2-a=-(r-2)$. Thus, by \cref{lemma3}, 
it follows that 
\begin{align*}
   L''(x,x+r-1) \in \cF(r,-(r-2)) \setminus \Qu, 
\end{align*}
and since $L''$ preserves the first coordinate, we obtain
\begin{align*}
  \Sm\big(r-1,\z(-(r-2))\big) 
  = \Sm\big(r-1,r-1\big) \subseteq \Sm\big(r,-(r-2)\big).  
\end{align*}

If there was a point $(a,a-r+2) \in \cF(r,-(r-2)) \setminus \Qu$ with $a \not\in [1-r,-1] \cap \ZZ$, \cref{lemma3} yields $a < 1-r$. Then, by \cref{lemma3} and induction, there exists $a' \in \ZZ$ such that $(a',a-r+2) \in \partial \scBpc(O,r-1)$ with $L'(a',a-r+2)=(a,a-r+2)$, or there exists $b' \in \ZZ$ such that $(a,b') \in \partial \scBpc(O,r-1)$ with  $L''(a,b') = (a,a-r+2)$. 

In the former case, we have $a'=a-2r+5$, and hence $(a',a-r+2) \in \cF(r-1,r-3)$.
But then we obtain $3-2(r-1) \leq a < 3-2r$, by induction, a contradiction. 

In the latter case, we have $b' = a+r-1$, 
and hence $(a,b') \in \cF(r-1,r-1)$, which contradicts \cref{lemma3}. 

Now, fix $c \in \ZZ$ with $|c| < r - 2$ and $c \equiv r \pmod 2$. It is sufficient to present the proof in the case $c\le 0$, 
as the case where $c>0$ can be treated in an analogous manner or, differently, the result can be obtained by applying symmetry.
Thus, in the following we assume that~$c \leq 0$. 
For each point
$\big(a,a+\z(c)\big) \in \cF\big(r-1,\z(c)\big) \setminus \Qu$, we have 
\begin{align*}
   L''\big(a,a+\z(c)\big) = (a,a-\z(c)+1) = (a,a+c).
\end{align*}
We claim that this implies 
\begin{align*}
  L''\big(a,a+\z(c)\big) \in \cF(r,c) \setminus \Qu.
\end{align*}
It suffices to show that $L''(x,x+\z(c)) \not\in \cF(r-2,c) \setminus \Qu$. 
Since $|c| = r - 2k = (r-2)-2(k-1)$, by \cref{characterization} we obtain:
\begin{align*}
  \Sm(r-2,c) 
  = \begin{cases}
  \big([0,r-2-k] \cap \ZZ\big) + c(k-1) - \binom{k}{2}, 
       &\text{ if } k \equiv 1 \pmod 2;\\[8pt]
  \big([k+1-r,-1] \cap \ZZ\big) + c(k-2) - \binom{k-1}{2}, 
       &\text{ if } k \equiv 0 \pmod 2.
    \end{cases}
\end{align*}
Since $|\z(c)| = |c| + 1 = r - 2k + 1 = (n-1)-2(k-1)$ and $\z(c) = 1-c$, \cref{characterization} also yields:
\begin{align*}
  \Sm(r-1,\z(c)) 
  = \begin{cases}
    \big([k-r,-1] \cap \ZZ\big)+(c-1)(k-1)-\binom{k-1}{2}, 
    &\text{ if } k \equiv 1 \pmod 2;\\[8pt]
    \big([0,r-1-k] \cap \ZZ\big)+(c-1)k - \binom{k}{2},
    &\text{ if } k \equiv 0 \pmod 2.
    \end{cases}
\end{align*}
Observing that
\begin{align*}
  -1 + (c-1)(k-1) - \binom{k-1}{2} 
  - \left(0 + c(k-1) - \binom{k}{2}\right) 
     = -1 < 0,
\end{align*}
    and
\begin{align*}
  r-1-k+(c-1)k-\binom{k}{2} 
  - \left(k+1-r+c(k-2) - \binom{k-1}{2}\right) 
  = -1 < 0,
\end{align*}
the claim follows. 

Thus, because $L''$ preserves the first coordinate, 
$\Sm\big(r-1,\z(c)\big) \subseteq \Sm(r,c)$. 
Note that more generally, for $1 \leq j \leq k-1$, 
\cref{characterization} yields:
\begin{align*}
  \Sm(r-2j,c) 
  = \begin{cases}
    \big[k+j-r,-1\big]+c(k-j-1)-\binom[0.81]{k-j}{2},  &\text{ if } j \not\equiv k \pmod 2;\\[8pt]
    \big[0,r-j-k-1\big]+c(k-j)-\binom[0.81]{k-j+1}{2}, &\text{ if } j \equiv k \pmod 2,
  \end{cases}
\end{align*}
since $|c| = r-2k = (n-2j)-2(k-j)$. 

Also note that for every integer $j$ with $1 \leq j \leq k - 2$, we have 
\begin{align*}
  r-j-k-1+c(k-j)-&\binom[0.81]{k-j+1}{2}+1 \\
  &= k+(j+1)-r+c(k-(j+1)-1)-\binom[0.81]{k-(j+1)}{2},
\end{align*}
and
\begin{align*}
  -1+c(k-j-1)-\binom[0.81]{k-j}{2}+1 
  &= c(k-(j+1)) - \binom[0.81]{k-(j+1)+1}{2}. 
\end{align*}
Thus, if $k \equiv 1\pmod 2$, we obtain
\begin{align*}
   \left[c(k-1)-\binom{k}{2},-c-1\right] \cap \ZZ 
   &= \left(\left[c(k-1)-\binom{k}{2},-1\right] \cap \ZZ\right) \sqcup \left([0,-c-1] \cap \ZZ\right) \\
   &= \bigsqcup_{j=1}^{k} \Sm(r-2j,c),
\end{align*}
and, if $k \equiv 0 \pmod 2$, we obtain
\begin{align*}
        \left[k+1-r+c(k-2)-\binom{k-1}{2},-c-1\right] \cap \ZZ = \bigsqcup_{j=1}^{k} \Sm(r-2j,c). 
\end{align*}

Next, we need to distinguish two cases, depending on the parity of $k$.
Since the analysis with the appropriate modifications is similar in both cases, 
for the rest of the proof, we will assume that $k \equiv 0 \pmod 2$.

If there was a point $(a,a+c) \in \cF(r,c) \setminus \Qu$ 
such that $(a,a+c) \not\in \Sm\big(r-1,\z(c)\big)$, 
then $a < (c-1)k-\binom{k}{2}$, by the above.
Noting that $|c-1| = r-2k+1 = (n-1)-2(k-1)$, \cref{characterization} yields
\begin{align*}
  \Sm(r-1,c-1) = \big([k-r,-1]\cap \ZZ\big) + (c-1)(k-2) - \binom[0.81]{k-1}{2}.
\end{align*}
Since $r-2k = |c| \geq 0$ implies $(c-1)k-\binom{k}{2}+1 \leq k-r+(c-1)(k-2)-\binom{k-1}{2}$,
it follows that $(a,a+c-1)\not\in \cF(r-1,c-1) \setminus \Qu$, and 
$(a+1,a+c-1) \not\in \cF(r-1,c-1) \setminus \Qu$.

Likewise, if $c \leq -1$, then $|c+1| = r-2k-1 = (r-1)-2k$. 
Therefore
\begin{align*}
  \Sm(r-1,c+1) = \big([0,r-k-2] \cap \ZZ\big) + (c+1)k - \binom[0.81]{k+1}{2},
\end{align*} 
by \cref{characterization}. Since $k \geq 1$ implies 
$(c-1)k-\binom{k}{2} + 1 \leq (c+1)k-\binom{k+1}{2}$, 
it follows that $(a-1,a+c), (a,a+c+1) \not\in \cF(r-1,c+1) \setminus \Qu$. 

If $c = 0$, we obtain $|c+1| = r-2k+1 = (r-1)-2(k-1)$, and hence 
\begin{align*}
 \Sm(r-1,c+1) = \big([0,r-1-k] \cap \ZZ\big) - (c+1)k - \binom[0.81]{k}{2},
\end{align*} 
once again by \cref{characterization}. 
Since $a < (c-1)k - \binom{k}{2} = -\binom{k+1}{2} = -(c+1)k - \binom{k}{2}$,
there exists $a^{*} \in \ZZ$ such that 
$(a^{*},a+c) \in \partial \scBpc\big(O,r-1\big)$,
and $L'(a^{*},a+c) = (a,a+c)$, 
or there exists $b^{*} \in \ZZ$ such that 
$(a,b^{*}) \in \partial \scBpc\big(O,r-1\big)$, 
and $L''(a,b^{*}) = (a,a+c)$. 

In the first case, we have $a^{*} = a+2c+1$, and $a+c-a^{*} = -c-1$. 

If $a^{*} \geq 0$, then by \cref{lemma3}, $r-1 = -c-1$.
Therefore, $r = |c|$, contradicting our assumption that $|c| \leq r-2$. 

If $a^{*} < 0$, we obtain $a^{*} \in \Sm(r-1,-c-1)$ yet by \cref{characterization}, we have:
\begin{equation*}
  \Sm(r-1,-c-1) 
  = \begin{cases}
  \big([0,r-k-2] \cap \ZZ\big)-(c+1)k-\binom[0.81]{k+1}{2},
           &\text{if } c \leq -1;\\[8pt]
  \big([0,r-1-k] \cap \ZZ\big)+(c+1)k-\binom[0.81]{k}{2},
        &\text{if } c = 0.
  \end{cases}   
\end{equation*}

If $c \leq -1$, then 
$k \geq 1$ implies $c(k+2)+1 \leq -(k+2)+1 \leq 0 \leq -(c+1)k$ 
and hence $a^{*} < (c-1)k - \binom{k}{2} + 2c + 1 \leq -(c+1)k-\binom{k+1}{2}$. 

If $c = 0$, then $k \geq 1$ now implies
$a^{*} < (c-1)k - \binom{k}{2} + 2c + 1 \leq k - \binom{k}{2} = (c+1)k - \binom{k}{2}$. 

In the second case, we have $b^{*} = a-c+1$. 
Then $a \in \Sm(r-1,-c+1)$. But, by \cref{characterization},
we find that
\begin{equation*}
  \Sm(r-1,-c+1) = \big([k-r,-1] \cap \ZZ\big) 
                    + (c-1)(k-1) - \binom[0.81]{k-1}{2},
\end{equation*}
and $(c-1)k-\binom{k}{2} = k-r+(c-1)(k-1)-\binom{k-1}{2}$. 

This completes the proof of Lemma~\ref{lemma7}.
\end{proof}

By \cref{lemma7}, \cref{characterization} yields parts \ref{Theorem0:II}\ref{Theorem0:IIb} and \ref{Theorem0:IIc} of \cref{Theorem0}. 

\section{Proof of Theorem~\ref{Theorem1}}
Recalling the notation $O=(0,0)$ for the origin, the boundaries of the smallest balls are:
$\partial \scBpc(O,0) = \{O\}$, and 
$\partial \scBpc(O,1) = \{(\pm 1, 0), (0, \pm 1)\}$,
as seen in Figure~\ref{FigureSevenBalls}.
Then formula~\eqref{eqTheorem1a} is verified for~$r=0$ and $r=1$, because
$5/2 - 1 + 5/2 = 4$.
We may therefore assume $r \geq 2$. 

For every $c \in \ZZ$,  writing $|c| = r-2k$, the following statements hold:
\begin{enumerate}
\setlength{\itemsep}{0pt}
    \setlength{\parskip}{8pt}
    \setlength{\parsep}{0pt} 
    \item 
$|\cF(r,c)| = 2r - 2k$, if $c \equiv r \pmod 2$ and $|c| \leq r - 2$,  
by formula~\eqref{eqChar2}; 
    \item 
$|\cF(r,c)| = 0$, if $c \not\equiv r \pmod 2$ and $|c| < r-2$, by \cref{lemma2};
    \item 
$|\cF(r,\pm r))| = \binom{r}{2} + 1 + r = \binom{r+1}{2} + 1$, by formula~\eqref{eqChar3}.
\end{enumerate}
Consequently, we may determine the number of points on the boundary 
$\partial \scBpc(O,r)$.
If~$r>0$ is even, we have:
\begin{align*}
  \#\partial \scBpc(O,r) 
    &= 2\left(\binom[0.81]{r+1}{2} + 1\right) + r 
    + 2\sum_{k=1}^{r/2-1}2r - 2k\\
    &= r^{2}+2r+2 + 4r\left(\sdfrac{r}{2}-1\right) 
    - 4\binom[0.81]{r/2}{2}\\
    &= \sdfrac{5r^{2}}{2}-r+2.
\end{align*}
Likewise, if $r$ is odd, we have:
\begin{align*}
  \#\partial \scBpc(O,r) &= 2\left(\binom[0.81]{r+1}{2} + 1\right) + 2\sum_{k=1}^{(r+1)/2-1}2r - 2k\\
  &= r^{2}+r+2 + 4r\left(\sdfrac{r+1}{2}-1\right) 
     - 4\binom[0.81]{(r+1)/2}{2}\\
  &= \sdfrac{5r^{2}}{2}-r+\sdfrac{5}{2}.
\end{align*}
Taking the sum over the interval $[0,R] \cap \ZZ$ of the expressions ~\eqref{eqTheorem1a} obtained for $\#\partial \scBpc(0,r)$, where $0 \leq r \leq R$, we obtain the area of the 
the ball $\scBpc(O,R)$:
\begin{align*}
  \sum_{r=0}^R\#\partial\scBpc(O,r)
  &= \#\partial\scBpc(O,0)
  + \sum_{\substack{r=1\\r \text{ odd}}}^R
      \#\partial\scBpc(O,r)
  + \sum_{\substack{r=2\\r \text{ even}}}^R
      \#\partial\scBpc(O,r)\\
  &=  \sum_{\substack{r=1\\r \text{ odd}}}^R
      \bigg(\sdfrac{5r^{2}}{2}-r+\sdfrac{5}{2}\bigg)
   + \sum_{\substack{r=2\\r \text{ even}}}^R
      \#\partial\scBpc(O,r)  
   \bigg( \sdfrac{5r^{2}}{2}-r+2\bigg).
\end{align*}
It then follows that:
\begin{align*}
  \#\scBpc(O,R) 
  &= 1 + \sum_{r=1}^R \bigg( \sdfrac{5r^{2}}{2}-r+2\bigg)
    +\sdfrac 12\sum_{\substack{r=1\\r \text{ odd}}}^R 1\\
  &= 1 + \sdfrac{5R(R+1)(R+2)}{12}  
   - \sdfrac{R(R+1)}{2}  +2R 
   +\sdfrac 12 \left\lceil\sdfrac{R}{2}\right\rceil\\
  &=  \sdfrac{1}{12}\big(10R^3 + 9R^2 + 23R\big) 
         + \sdfrac{1}{2}\left\lceil\sdfrac{R}{2}\right\rceil + 1.
\end{align*}  
 
This concludes the proof of \cref{Theorem1}.

\bigskip
\noindent
\textbf{Acknowledgements.}
The authors acknowledge Omer Cantor who sent them his proof of parts~\ref{Problem2} and~\ref{Problem3}
of Problem~\ref{Problem}
after the first draft of the manuscript became \mbox{public}.



\end{document}